\documentclass[a4paper]{amsart}
\usepackage[all]{xy}
\usepackage[colorlinks=true]{hyperref}
\usepackage{enumerate}
\usepackage{amssymb}
\usepackage{comment}
\newtheorem{thm}{Theorem}[section]

\newtheorem{prop}[thm]{Proposition}
\newtheorem{lem}[thm]{Lemma}
\newtheorem{cor}[thm]{Corollary}

\theoremstyle{definition}
\newtheorem{defn}[thm]{Definition}
\newtheorem{ex}[thm]{Example}

\theoremstyle{remark}
\newtheorem{rem}[thm]{Remark}

\newcommand{\Rr}{\mathbb R}

\renewcommand{\d}{\mathrm d}                           

\newcommand{\set}[1]{\left\{#1\right\}}                
\newcommand{\eval}[1]{\left\langle#1\right\rangle}     
\newcommand{\br}{\left[\, , \, \right]}                
\newcommand{\brr}[1]{\left[#1\right]}                  
\newcommand{\cbrr}[1]{\left[\! \left[#1\right]\!\right]}

\newcommand{\ds}{\displaystyle}                         
\newcommand{\p}{\partial}
\newcommand{\ran}{\rangle}
\newcommand{\lan}{\langle}
\newcommand{\lie}{\mathcal{L}}

\newcommand{\C}{\ensuremath{\mathcal{C}}}  
\newcommand{\T}{{\mathcal{T}}}             
\renewcommand{\d}{\mathrm d}               
\newcommand{\Lie}{\boldsymbol{\pounds}}    

\newcommand{\smc}{\mbox{\,\tiny{$\circ $}\,}}         

\DeclareMathOperator{\graf}{graph}          

\newcommand{\al}{\alpha}
\newcommand{\be}{\beta}

\begin{document}

\title[On Jacobi quasi-Nijenhuis algebroids and Courant-Jacobi morphisms]
{On Jacobi quasi-Nijenhuis algebroids and Courant-Jacobi algebroid
morphisms}

\author{Raquel Caseiro}
\address{CMUC, Department of Mathematics, University of Coimbra, Portugal}
\email{raquel@mat.uc.pt}
\author{Antonio De Nicola}
\address{CMUC, Department of Mathematics, University of Coimbra, Portugal}
\email{antondenicola@gmail.com}
\author{Joana M. Nunes da Costa}
\address{CMUC, Department of Mathematics, University of Coimbra, Portugal}
\email{jmcosta@mat.uc.pt}

\begin{abstract}
We propose a definition of Jacobi quasi-Nijenhuis algebroid and
show that any such Jacobi algebroid has an associated quasi-Jacobi
bialgebroid. Therefore, also an associated Courant-Jacobi
algebroid is obtained. We introduce the notions of  quasi-Jacobi
bialgebroid morphism and  Courant-Jacobi algebroid morphism
providing also some examples of Courant-Jacobi algebroid
morphisms.
\end{abstract}

\maketitle


\


\section*{Introduction}             %
\label{sec:introduction}           %
\ \ \ \ It is well-known that one can consider, on a Lie
algebroid, some additional geometric tools that provide richer
structures such as, among others, Poisson Lie algebroids,
Poisson-Nijenhuis Lie algebroids \cite{KosMagri, GraUrb}, Lie
bialgebroids \cite{LiuWeiXu}
 and quasi-Lie bialgebroids
\cite{Roy}. This procedure can be repeated if we start with a Jacobi
algebroid $(A, \phi)$, i.e. a Lie algebroid $A$ together with a
$1$-cocycle $\phi \in \Gamma(A^\ast)$. The  notion of Jacobi
algebroid was introduced in \cite{GraMar1} and \cite{IglMarr} and it
is the one that fits in the Jacobi framework. As in the Poisson
case, one can define Jacobi bialgebroids \cite{GraMar1,IglMarr},
quasi-Jacobi bialgebroids \cite{jf-qjjq}, Jacobi-Nijenhuis
algebroids \cite{rj} and so on. In the case of a Poisson-Nijenhuis
Lie algebroid $(A, \pi, N)$, $\pi$  is a Poisson bivector on $A$
which is compatible, in a certain sense, with the Nijenhuis operator
$N$, i.e. a vector bundle map $N : A \to A$ with vanishing torsion.
Relaxing the condition which imposes the vanishing of the torsion,
and admitting that its nonzero value depends on a certain closed
$3$-form, one obtains what is called a Poisson quasi-Nijenhuis Lie
algebroid, a notion recently introduced in \cite {StiXu} for the
case of manifolds and then extended to general Lie algebroids in
\cite{Ant}. It is worthwhile to mention that the notion considered
in \cite{Ant} is even more general since the compatibility between
the Poisson bivector and the operator $N$ is less restrictive,
leading to the so-called Poisson quasi-Nijenhuis Lie algebroid with
background. Either in the Poisson quasi-Nijenhuis case or in the
Poisson quasi-Nijenhuis with background case, there is an associated
quasi-Lie bialgebroid. Since the double of a quasi-Lie bialgebroid
is a Courant algebroid, one has a Courant algebroid associated to
each Poisson quasi-Nijenhuis Lie algebroid (with background).

In the Jacobi setting, the notion of Jacobi-Nijenhuis algebroid was
introduced in \cite{rj} and it is a quadruple $(A, \phi, \pi, N)$
where $(A, \phi)$ is a Jacobi algebroid, $\pi$ is a Jacobi bivector
(i.e., $[\pi, \pi]^{\phi}=0$) that is compatible with the operator
$N$ which has vanishing torsion with respect to the bracket $[ ,
]^{\phi}$. Admitting a nonzero torsion that depends on a certain
$\d^{\phi}$-closed $3$-form, we define a Jacobi quasi-Nijenhuis
algebroid and show that it has an associated quasi-Jacobi
bialgebroid and therefore also an associated Courant-Jacobi
algebroid \cite{GraMar, jj}. Another concept we introduce, having a
relevant role in the paper, is that of quasi-Jacobi bialgebroid
morphism.

An important notion related to Courant and Courant-Jacobi algebroids
$A \to M$ is that of Dirac structure -- a special vector subbundle
$L \to M$ of $A \to M$ which inherits, by restriction, a Lie
algebroid structure. In \cite{AleXu, BurIglSev}, given a Courant
algebroid $A \to M$, the authors considered a vector subbundle $L
\to P$ over a submanifold  $P$ of $M$ and introduced the notion of
Dirac structure supported on a submanifold, which is obviously a
generalization of a Dirac structure. In a natural way, we extend the
definition to the case of Dirac structures of Courant-Jacobi
algebroids supported on submanifolds of the base manifold and, using
this notion,  we introduce the concept of Courant-Jacobi algebroid
morphism. We prove that a Courant-Jacobi algebroid morphism is
associated to any morphism of quasi-Jacobi bialgebroids, as well as
to certain Jacobi quasi-Nijenhuis structures.

Very recently, while we were finishing this paper, the notion of
Jacobi quasi-Nijenhuis algebroid was also introduced in \cite{Sheng}
where the author claims the equivalence between a Jacobi
quasi-Nijenhuis structure on a Jacobi  algebroid and  a quasi-Jacobi
bialgebroid structure on its dual. The author proves his result just
on functions and exact 1-forms but these in general do not generate
the space of 1-forms in Jacobi (or Lie) algebroids so that the proof
is not correct. Only in the case of the tangent bundle to a manifold
does the equivalence hold.

This paper is divided into two sections. In section 1 we introduce
the notion of quasi-Jacobi bialgebroid morphism, we give a simpler
definition of Courant-Jacobi algebroids and we extend to
Courant-Jacobi algebroids the notion of Dirac structure supported on
a submanifold of the base manifold, which was initially given in
\cite{AleXu, BurIglSev} for Courant algebroids morphisms. Finally,
we define Courant-Jacobi morphisms and show that in the case where
the Courant-Jacobi algebroids are doubles of quasi-Jacobi
bialgebroids, a morphism $(\Phi, \phi)$ of Courant-Jacobi algebroids
provides an example of a Dirac structure supported on graph$\phi$.
In section 2, after introducing the concept of Jacobi
quasi-Nijenhuis algebroid, we show that it has an associated Poisson
quasi-Nijenhuis Lie algebroid. The main result of this section
states that to each Jacobi quasi-Nijenhuis algebroid one can
associate a quasi-Jacobi bialgebroid and in the case where the Lie
algebroid is the tangent bundle to a manifold, one has a one-to-one
correspondence. For a special kind of Jacobi quasi-Nijenhuis
algebroid, we show  that the graph of the quasi-Nijenhuis operator
determines a Courant-Jacobi algebroid morphism. \ \ \ \

\section{Quasi-Jacobi bialgebroids morphisms}%
\label{sec:Quasi-Jacobi bialgebroids}
\subsection{Quasi-Jacobi bialgebroids}

Recall that a \emph{Jacobi algebroid} \cite{GraMar1} or
\emph{generalized Lie algebroid} \cite{IglMarr} is a pair $(A,
{\phi})$ where  $A=(A, \br, \rho)$ is a Lie algebroid over a
manifold $M$ and  ${\phi}\in\Gamma(A^*)$ is a 1-cocycle, i.e.
${\ds \d {\phi}=0}$. A Jacobi algebroid has an associated
Schouten-Jacobi bracket on the graded algebra
$\Gamma(\wedge^\bullet A)$ of multivector fields on $A$ given by
\begin{equation}\label{Jacobibracket}
    \brr{P,Q}^{\phi}=\brr{P,Q}+(p-1)P\wedge i_{\phi} Q- (-1)^{p-1}(q-1)
    i_{\phi} P\wedge Q,
\end{equation}
for $P\in \Gamma(\wedge^p A), Q\in\Gamma(\wedge^q A)$, $p\geq 1, q
\geq 1$.

In a Jacobi algebroid the anchor $\rho$ is replaced by
$\rho^{\phi}$ which is the representation of the Lie algebra
$\Gamma(A)$ on $C^\infty(M)$ given by
$$\rho^{\phi}(X)f=\rho(X)f+f\lan {\phi},X\ran.$$ The cohomology
operator $\d^{\phi}$ associated with this representation is called
the \emph{${\phi}$-differential} of $A$ and is given by
\begin{equation}\label{phidiff}
 \d^{\phi} \omega= \d\omega+{\phi}\wedge
 \omega,\quad\omega\in\Gamma(\wedge^k A^*).
\end{equation}

Any vector bundle map $\Psi:A\to B$ induces a map $\Psi^*:
\Gamma(B^*)\to \Gamma(A^*)$ which assigns to each section
$\al\in\Gamma(B^*)$
 the section $\Psi^*\al$ given by
\[
\Psi^*\al(X)(m)=\eval{\al(\psi(m)), \Psi_m X(m)}, \quad \forall m\in
M, \, X\in \Gamma(A),
\]
where $\psi:M\to N$ is  the map induced by $\Psi$ on the base
manifolds. We denote by the same symbol $\Psi^*$ the extension of
this map to the multisections of $B^*$, where we set $\Psi^*f=f\smc
\psi$, for $f\in C^\infty (N)$.

Now, let $A\to M$ and $B\to N$ be two Lie algebroids and
${\phi_A}\in\Gamma(A^*)$, ${\phi_B}\in\Gamma(B^*)$  1-cocycles, so
that $(A, {\phi_A})$ and $(B, {\phi_B})$ are Jacobi algebroids.
Then it is natural to give the following definition.
\begin{defn}\label{def:Jacobi:morphism}
A \emph{Jacobi algebroid morphism} from $(A, {\phi_A})$ to $(B,
{\phi_B})$ is a vector bundle map $\Psi:A\to B$ such that
$\Psi^*:(\Gamma(\wedge^\bullet B^*), \d^{\phi_B}_B)\to
(\Gamma(\wedge^\bullet A^*),\d^{\phi_A}_A)$ is a chain map.
\end{defn}

The definition of Jacobi algebroid morphism was already given in
\cite{Igl} in terms of Lie algebroid morphisms. The two definitions
are equivalent, as is stated in the following proposition.
\begin{prop}\label{prop:Jacobi:morphism}
Let $(A, {\phi_A})$ and $(B, {\phi_B})$ be Jacobi algebroids and
$\Psi:A\to B$ a vector bundle map. Then  $\Psi$ is a Jacobi
algebroid morphism if and only if the following conditions are
satisfied:
\begin{itemize}
 \item [i)] $\Psi^*\colon(\Gamma(\wedge^\bullet B^*), \d_B)\to
(\Gamma(\wedge^\bullet A^*),\d_A)$ is a chain map, i.e. $\Psi$ is
a Lie algebroid morphism;
 \item [ii)]$\Psi^*{\phi_B}={\phi_A}$.
 \end{itemize}
\end{prop}
\begin{proof}
If conditions i) and ii) hold, then $\Psi$ is clearly a Jacobi
algebroid morphism. Conversely, if $\Psi$ is a Jacobi algebroid
morphism, then by evaluating the condition
\[\Psi^* \smc \d^{\phi_B}_B=\d^{\phi_A}_A \smc \Psi^*\]
over the constant function $1\colon N \to \Rr$ one obtains
condition ii). Then, condition i) follows as well.
\end{proof}

In the Jacobi framework, the structure analogous to that of
quasi-Lie bialgebroid is given by the following definitions.

\begin{defn}
Let $(A, \phi)$ be a Jacobi algebroid and $W\in \Gamma(A)$. Then a
(degree-one) \emph{$W-$quasi differential} on $(A, \phi)$ is a
linear operator $\tilde \d_*\colon \Gamma\left(\wedge^\bullet
A\right) \to \Gamma\left(\wedge^{\bullet+1} A\right)$ such that
\begin{itemize}
 \item [i)] $\tilde \d_* (P \wedge Q)= \tilde \d_* P \wedge Q + (-1)^p P \wedge \tilde \d_* Q  - W \wedge P \wedge Q,$
 \item [ii)]$\tilde \d_*\brr{P,Q}^{\phi}_A=[\tilde \d_* P,Q]^{\phi}_A + (-1)^{p+1}[P, \tilde \d_*Q]^{\phi}_A,$
\end{itemize}
 for any $P\in \Gamma(\wedge^p A), Q\in\Gamma(\wedge^q A)$.
\end{defn}
If $\tilde \d_*$ is a $W-$quasi differential on $(A, \phi)$, then
considering the constant function $1\colon M \to \Rr$, from i) we
get $\tilde \d_*1=W$. So $\d_*P=\tilde\d_*P-W\wedge P$ is the
associated  degree one derivation of the  Gerstenhaber algebra
$\left(\Gamma(\wedge^\bullet A), \wedge, \br_A\right)$.

\begin{defn}\label{def:quasi:Jacobi:bialgebroid}
A \emph{quasi-Jacobi bialgebroid} is a Jacobi algebroid $(A,
\phi)$ equipped with a $W-$quasi differential $\tilde \d_*$ and a
$3$-section of $A$, $X_A\in\Gamma(\wedge^{3} A)$ such that
\[\tilde \d_*X_A=0 \quad \mbox{and} \quad \tilde \d_*^{2}=\brr{X_A, - }^{\phi}_A.\]
\end{defn}

\begin{rem}
This definition is equivalent to that of \cite{jf-tj}. Given a
quasi-Jacobi bialgebroid $(A,\phi_A,\tilde\d_{*},X_A)$, by using
$\d_{*}$ one can define a map $\rho_{A^*}\colon A^* \to TM$ by
\[
\rho_{A^*}(\al)(f)=\al(\d_{*}f)
\]
and a bracket $\br_{A^*}$ on the sections of $A^*$ by
\[
\brr{\al,\be}_{A^*}(X)=\rho_{A^*}(\al)(\be(X))-\rho_{A^*}(\be)(\al(X))-\d_{*}X(\al,\be),
\]
for $\al$, $\be\in \Gamma(A^*)$, $f\in C^\infty(M)$ and
$X\in\Gamma(A)$. Then, the Jacobi bracket $\br_{A^*}^W$ on
$\Gamma(A^*)$ is defined in the same way as 
in \eqref{Jacobibracket}.

 In particular, when $\phi=0$ and $W=0$, then $(A,\d_{*},X_A)$ is a
\emph{quasi-Lie bialgebroid} \cite{Roy}. On the other hand, when
$X_A=0$ then a  Jacobi bialgebroid $(A,\phi,\tilde\d_{*})$ is
defined.
\end{rem}

Examples of quasi-Jacobi bialgebroids are given by the
quasi-Jacobi bialgebroid associated to a twisted-Jacobi manifold
\cite{jf-tj,jf-qjjq}.

We propose the following definition of morphism between
quasi-Jacobi bialgebroids:

\begin{defn}\label{def:quasi:Jacobi:morphism}
Let $(A,\phi_A,\tilde\d_{*A},X_A)$ and
$(B,\phi_B,\tilde\d_{*B},X_B)$ be quasi-Jacobi bialgebroids over $M$
and $N$, respectively. A vector bundle map $\Psi:A\to B$ is  a
\emph{quasi-Jacobi bialgebroid morphism }if
\begin{itemize}
\item[1)] $\Psi$ is a Jacobi algebroid morphism;
\item[2)] $\Psi^*$ is compatible with the brackets on the sections of $A^*$
and $B^*$:
          \[\brr{\Psi^*\al,
          \Psi^*\be}_{A^*}=\Psi^*\brr{\al,\be}_{B^*};\]
\item[3)] the vector fields $\rho_{B^*}(\al)$ and
$\rho_{A^*}(\Psi^*\al)$ are $\psi$-related:
           \[T\psi\cdot \rho_{A^*}(\Psi^*\al)=\rho_{B^*}(\al)\smc \psi;\]
\item[4)] $ \Psi X_A=X_B\smc\psi$;
\item[5)] $ \Psi W_A=W_B\smc\psi$;
\end{itemize}
where $\al$, $\be\in \Gamma(B^*)$ and $\psi:M\to N$ is the smooth
map induced by $\Psi$ on the base and
$W_A=\tilde\d_{*A}1,W_B=\tilde\d_{*B}1$ are the associated
sections of $A^*,B^*$, respectively.
\end{defn}

\begin{ex}
When $X_A=X_B=0$ we have a \emph{Jacobi bialgebroid morphism}, i.e.,
 a Jacobi algebroid morphism $\Psi:A\to B$ such that conditions 2) and 3) above are satisfied
and $\Psi W_A=W_B\smc \psi$.
\end{ex}

\begin{ex}
Let $(A,\phi_A,\tilde\d_{*A},X_A)$ and
$(B,\phi_B,\tilde\d_{*B},X_B)$ be quasi-Jacobi bialgebroids over the
same base manifold $M$. We can see that a base-preserving
quasi-Jacobi bialgebroid morphism (i.e. such that
$\psi=\textrm{id}$) is a vector bundle map $\Psi:A\to B$ such that
$\Psi^* \smc \d^{\phi_B}_B = \d^{\phi_A}_A \smc \Psi^*$, $\Psi \smc
\tilde\d_{*B}=\tilde\d_{*A} \smc \Psi$ and $\ds \Psi X_A=X_B$.
\end{ex}

In Definition \ref{def:quasi:Jacobi:morphism}, when $\phi=0$ and
$W=0$, then $(A,\d_{*A},X_A)$ and $(B,\d_{*B},X_B)$ are quasi-Lie
bialgebroids and we obtain as a special case the definition of a
\emph{quasi-Lie bialgebroid morphism} as a vector bundle map
$\Psi:A\to B$ such that $\Psi$ is a Lie algebroid morphism and
satisfies conditions 2)-4).

\subsection{Courant-Jacobi algebroids}
Next, following the ideas of \cite{ks3}, we will present a
definition of  Courant-Jacobi algebroid which is simpler than the
original one in \cite{GraMar}.
\begin{defn}\label{def:Courant-Jacobi}
A \emph{Courant-Jacobi algebroid} is a vector bundle $E\to M$
equipped with a fiberwise nondegenerate symmetric bilinear form
$\eval{ \, , \, }$, a vector bundle map $\rho: E\to TM\oplus \Rr$
and a bilinear bracket $\smc$ on $\Gamma(E)$ satisfying:
 \begin{itemize}
 \item[CJ1)] $\ds e_1\smc(e_2\smc e_3)=(e_1\smc e_2)\smc e_3+e_2\smc(e_1\smc
 e_3)$;
 \item[CJ2)] $\ds \Lie_{\rho(e_1)}\eval{e_2, e_3}=\eval{e_1, e_2 \smc e_3 + e_3 \smc
 e_2}$;
 \item[CJ3)] $\ds \Lie_{\rho(e_1)}\eval{e_2, e_3}=\eval{e_1 \smc e_2,  e_3} + \eval{e_2, e_1 \smc
 e_3}$;
 \end{itemize}
for all $e_1,e_2,e_3\in\Gamma(E)$.
\end{defn}

Recall that in the definition of Courant-Jacobi algebroid in
\cite{GraMar}, instead of CJ1)-CJ3) one has the following
conditions:
 \begin{itemize}
 \item[CJ$'$1)] $\ds e_1\smc(e_2\smc e_3)=(e_1\smc e_2)\smc e_3+e_2\smc(e_1\smc
 e_3)$;
 \item[CJ$'$2)] $\ds \rho(e_1\smc e_2)=\brr{\rho(e_1),\rho(e_2)}$\footnote{The bracket on the right
 hand side is the Lie bracket on $\Gamma(TM\oplus \Rr)$ defined by
 $\brr{(X,f),(Y,g)}=(\brr{X,Y},X\cdot g-Y\cdot f)$};
 \item[CJ$'$3)] $\ds \eval{e_1\smc e, e}=\eval{e_1, e \smc e}$;
 \item[CJ$'$4)] $\ds \Lie_{\rho(e_1)}\eval{e,e}=2\eval{e_1\smc e,
 e}$;
 \end{itemize}
for all $e,e_1,e_2,e_3\in\Gamma(E)$.

The next proposition proves that Definition \ref{def:Courant-Jacobi}
is equivalent to the one in \cite{GraMar}.

\begin{prop}\label{prop:equivalence:Courant-Jacobi}
Let $E\to M$ be a vector bundle  equipped with a fiberwise
nondegenerate symmetric bilinear form $\eval{ \, , \, }$, a vector
bundle map $\rho: E\to TM\oplus \Rr$ and a bilinear bracket $\smc$
on $\Gamma(E)$. Then, the conditions \emph{CJ1)-CJ3)} are equivalent
to \emph{CJ$'$1)-CJ$'$4)}.

\end{prop}
\begin{proof}
Conditions CJ$'$1) and  CJ1) are the same. Let us first show that
conditions CJ$'$2)-CJ$'$4) imply conditions CJ2) and CJ3). For any
$e_1,e_2,e_3\in\Gamma(E)$ we have from condition CJ$'$4) and the
bilinearity of $\smc$ and $\eval{ \, , \, }$
\begin{equation}\label{rhocj4}
\Lie_{\rho(e_1)}\eval{e_2+e_3,e_2+e_3}
                   =2\eval{e_1\smc e_2,e_2}+2\eval{e_1\smc e_2,e_3}+2\eval{e_1\smc
                    e_3,e_2}+2\eval{e_1\smc e_3,e_3}.
\end{equation}
On the other hand, we also have
\begin{equation}\label{rhobil}
\Lie_{\rho(e_1)}\eval{e_2+e_3,e_2+e_3}=\Lie_{\rho(e_1)}\eval{e_2,e_2}+2\Lie_{\rho(e_1)}\eval{e_2,e_3}
+\Lie_{\rho(e_1)}\eval{e_3,e_3}.
\end{equation}
 From
\eqref{rhocj4} and \eqref{rhobil}, and using again condition
CJ$'$4), we obtain CJ3). Condition CJ$'$3) allows us to write
\[
\eval{e_1\smc (e_3+e_2), e_3+e_2}=\eval{e_1, (e_3+e_2) \smc
(e_3+e_2)},
\]
or, equivalently,
\[
\eval{e_1\smc e_3,e_2} + \eval{e_1\smc e_2,e_3} =\eval{e_1, (e_2
\smc e_3) + (e_3 \smc e_2)},
\]
which expresses the equality of the right-hand sides of equations
CJ2) and CJ3) and therefore CJ2) holds.

Now let us prove that, conversely, CJ2) and CJ3) imply
CJ$'$2)-CJ$'$4). If we take $e_3=e_2$ in CJ3), we obtain
\begin{equation}\label{cj'3z=e_2}
\Lie_{\rho(e_1)}\eval{e_2, e_2}=2\eval{e_1 \smc e_2,  e_2},
\end{equation}
which is exactly CJ$'$4). Moreover, if we take $e_3=e_2$ in CJ2), we
get
\begin{equation}\label{cj'2z=e_2}
\Lie_{\rho(e_1)}\eval{e_2, e_2}=2\eval{e_1,  e_2 \smc  e_2}.
\end{equation}
From \eqref{cj'2z=e_2} and \eqref{cj'3z=e_2}, we conclude that
CJ$'$3) holds.

It remains to show that $\rho(e_1\smc e_2)=
\brr{\rho(e_1),\rho(e_2)}$, for all $e_1,e_2\in\Gamma(E)$. For any
$e_1\in\Gamma(E)$, $\rho(e_1)$ is a first order differential
operator on $M$. Then, $\Lie_{\rho(e_1)}(1)=\theta(e_1)$ with
$\theta$ a 1-cocyle  and, using  CJ3), we have:
\begin{equation}\label{theta}
e_1 \smc fe_2 = f e_1 \smc e_2 +
(\Lie_{\rho(e_1)}f-f\theta(e_1))e_2),
\end{equation}
for any $f\in C^\infty(M)$.
Starting with the Jacobi identity and
applying \eqref{theta}, we get
\begin{equation}\label{ele}
\begin{split}
e_1 \smc (e_2\smc fe_3) = (e_1\smc e_2)\smc fe_3+ e_2\smc (e_1\smc
fe_3) = f(e_1\smc
e_2) \smc e_3+ (\Lie_{\rho(e_1\smc e_2)}f)e_3-\theta(e_1 \smc e_2) fe_3\\
+ f e_2 \smc(e_1\smc e_3)+(\Lie_{\rho(e_2)}f)(e_1\smc e_3) -
\theta( e_2)f e_1\smc e_3
+(\Lie_{\rho(e_1)}f)(e_2\smc e_3) + \Lie_{\rho(e_2)}(\Lie_{\rho(e_1)}f)e_3\\
-\theta(e_2)(\Lie_{\rho(e_1)}f)e_3 - \theta(e_1)f e_2\smc e_3 -
\Lie_{\rho(e_2)}(\theta(e_1)f)e_3+ \theta(e_2)\theta(e_1)fe_3 .
\end{split}
\end{equation}
On the other hand,
\begin{equation}\label{twe}
\begin{split}
e_1 &\smc (e_2\smc fe_3)= e_1\smc (f e_2\smc e_3 +
(\Lie_{\rho(e_2)}f)e_3 -
\theta(e_2)fe_3)\\
&=fe_1\smc (e_2\smc e_3) + (\Lie_{\rho(e_1)}f)e_2\smc e_3 -
\theta(e_1)f e_2\smc e_3
+ (\Lie_{\rho(e_2)}f) e_1\smc e_3 + \Lie_{\rho(e_1)}(\Lie_{\rho(e_2)}f)e_3 \\
&-\theta(e_1)(\Lie_{\rho(e_2)}f)e_3 - \theta(e_2)f e_1\smc e_3 -
\Lie_{\rho(e_1)}(\theta(e_2)f)e_3+ \theta(e_1)\theta(e_2)fe_3 .
\end{split}
\end{equation}
From \eqref{ele} and \eqref{twe}, we obtain
\[
(\Lie_{\rho(e_1\smc e_2)}f) - \Lie_{\rho(e_1)}(\Lie_{\rho(e_2)}f)
+ \Lie_{\rho(e_2)}(\Lie_{\rho(e_1)}f) = 0
\]
and so,
\[
\rho(e_1\smc e_2)=\brr{\rho(e_1),\rho(e_2)}.
\]
\end{proof}

An alternate definition which was also proved to be equivalent to
that of Courant-Jacobi algebroid was given in \cite{jj} under the
name of \emph{generalized Courant algebroid}.

Associated with the bracket $\smc$, we can define  a
skew-symmetric bracket on the sections of $E$ by:
\[
\cbrr{e_1,e_2}=\frac{1}{2}\left(e_1\smc e_2- e_2\smc e_1\right)
\]
and the properties CJ1)-CJ3) can be expressed in terms of this
bracket.

\begin{rem}
Note that in Definition \ref{def:Courant-Jacobi} when
$\rho^*(0,1)=0$ (where
$(0,1)\in\Gamma(T^*M\times\Rr)=\Omega^{1}(M)\times C^\infty(M)$),
then by replacing $\rho$ with its natural projection
$\rho':=pr_{1}\smc\rho\colon E \to TM$ we obtain a Courant
algebroid structure on the vector bundle $E \to M$.
\end{rem}

\begin{ex}[Double of a
quasi-Jacobi bialgebroid
\cite{jf-tj}]\label{double:quasi:Jacobi:bialgebroid} Let
$(A,\phi,\tilde\d_*,X_A)$ be a quasi-Jacobi bialgebroid over $M$
and let $W\in\Gamma(A)$ be the associated section. Its double
$E=A\oplus A^*$ is a Courant-Jacobi algebroid when it is equipped
with the pairing $\eval{X+\al, Y+\be}=\ds{ \frac{1}{2}( \al(Y) +
\be(X))}$, the anchor $\rho=\rho^{\phi}_A+\rho^W_{A^*}$ and the
bracket
\begin{align*}
(X+\al) \smc (Y+\be) &= \left(\brr{X,Y}^{\phi}_A+ \tilde\Lie_{*\al} Y - i_\be \tilde\d_* X + X_A(\al,\be, - )\right) \\
&+ \left(\brr{\al,\be}^W_*+ \Lie^{\phi}_{X}\be - i_Y\d^{\phi}\al
\right),
\end{align*}
where $\Lie^{\phi}$ and $\tilde\Lie_*$ are the Lie derivative and
the quasi-Lie derivative operators defined, respectively, by
$\d^\phi$ and $\tilde\d_*$.
\end{ex}

Taking $X_A=0$ we have the Courant-Jacobi algebroid structure on
the double of a Jacobi bialgebroid. In particular, when the Jacobi
bialgebroid is $(TM\times\Rr, (0,1),\tilde\d_*=0)$, its double is
the standard Courant-Jacobi algebroid associated to the manifold
$M$, $\mathcal{E}^1(M)=(TM\times\Rr)\oplus(T^*M\times \Rr)$. The
 bilinear form reads:
\[
\eval{(X_1, f_1) + (\al_1, g_1), (X_2, f_2) + (\al_2, g_2)} =
\frac{1}{2} \left(i_{X_2}\al_1 + i_{X_1}\al_2 + f_1g_2 +
f_2g_1\right),
\]
the bilinear bracket is given by
\begin{align*}
& \left((X_1, f_1) +(\al_1, g_1)\right) \smc \left((X_2, f_2) +
(\al_2, g_2) \right) =
 \left([X_1,X_2], X_1(f_2) - X_2(f_1)\right)+\\
&+\left( \lie_{X_1}\al_2 - i_{X_2}d\al_1 +
 f_1\al_2  - f_2\al_1
+ g_2 df_1 + f_2 dg_1 ,\right.\\
&\left. X_1(g_2) - X_2(g_1)+ i_{X_2}\al_1  + f_1g_2\right),
\end{align*}
and the anchor is
$$\rho((X_1, f_1) + (\al_1,
g_1))=(X_1,f_1),$$ for  $\ds (X_i, f_i) + (\al_i, g_i) \in
\Gamma(\mathcal{E}^1(M))$, with $i = 1, 2$.

\subsection{Dirac structures supported on a submanifold}

A \emph{Dirac subbundle} or \emph{Dirac structure of a
Courant-Jacobi algebroid} $E$ is a vector subbundle $A\subset E$,
which is maximal isotropic with respect to the pairing $\eval{\, ,
\,}$ and is integrable in the sense that the space of the sections
of $A$ is closed under the bracket on $\Gamma(E)$. Restricting  the
skew-symmetric bracket of $E$ and the natural projection
$pr_{1}\smc\rho\colon E \to TM$ of the anchor $\rho$ to $A$, we
obtain a Lie algebroid structure on $A$.

As a way to generalize Dirac structures we have the concept of
generalized  Dirac structures or Dirac structures supported on a
submanifold of the base manifold. In the case of Courant algebroids
they were dealt with in \cite{AleXu, BurIglSev}. That definition can
be generalized to the Courant-Jacobi case as follows.

\begin{defn}
On a Courant-Jacobi algebroid $E\to M$, a \emph{Dirac structure
supported on a submanifold $P$ of $M$} or a \emph{generalized Dirac
structure} is a vector subbundle $F$ of $E_{|_P}$ such that:
 \begin{itemize}
 \item[D1)] for each $x\in P$, $F_x$ is maximal isotropic;
 \item[D2)] $F$ is compatible with the anchor, i.e.
         $\ds \rho_{|_P}(F)\subset TP \times \Rr$;
 \item[D3)]  For each $e_{1}, e_{2}\in \Gamma(E)$, such that ${e_{1}}_{|_P}$, ${e_{2}}_{|_P}\in \Gamma (F)$,
            we have $\ds ({e_{1}}\smc{e_{2}})_{|_P}\in \Gamma(F).$
 \end{itemize}
\end{defn}

Obviously, a Dirac structure supported on the whole base manifold
$M$ is  a Dirac structure in the usual sense.

Let $E\to M$ be a Courant-Jacobi algebroid and $L\to P$ a vector
subbundle of $A$ over a submanifold $P$ of $M$. We denote by
$L^\bot\to P$ the vector subbundle of $A$ over $P$ whose fiber over
any $x\in P$ is $L^\bot_x:=\set{e_x\in E_x|\; \forall e'_x\in L_x,
\; \eval{e_x, e'_x}=0}$. The conormal bundle $\nu^*(P)\to P$ is
defined by $\nu_x^*(P):=(T_xP)^0=\set{\al_x\in E^*_x|\; \forall
e_x\in L_x, \; \al_x(e_x)=0}$.

\begin{thm}\label{thm:dirac:Jacobi:direct:sum}
Let $E=A\oplus A^*$ be the double of a quasi-Jacobi bialgebroid
$(A,\phi,\tilde\d_*,X_A)$ over the  manifold $M$, $L\to P$ a
vector subbundle of $A$ over a submanifold $P$ of $M$, $F=L\oplus
L^\bot$. Then $F$ is a Dirac structure supported on $P$ if and
only if the following conditions hold:
\begin{itemize}
 \item[1)] $L$ is a Lie subalgebroid of $A$;
 \item[2)] $L^\bot$ is closed for the bracket on $A^*$ defined by $\d_*$;
 \item[3)] $L^\bot$ is compatible with the anchor, i.e., $\rho_{|_P}(L^\bot)\subset TP \times \Rr$;
 \item[4)] $X_{A|_P} \in \Gamma(\wedge^3 L)$.
 \end{itemize}
\end{thm}

\begin{proof}
Since $F=L\oplus L^\bot$, this  is  a Lagrangian subbundle of $E$.
If $F$ is a  Dirac structure supported on $P$, then we immediately
see that the stated four conditions are satisfied.

Conversely, suppose $L$ is a Lie subalgebroid of $A$,
$L^\bot\subset A^*$ is closed  for $\br_{A^*}$,
$\rho_{|_P}(L^\bot)\subset TP\times \Rr$ and
$X_{A|_P}\in\Gamma(\wedge^3 L)$. Obviously $F$ is  compatible with
the anchor, i.e, $\rho_{|_P}(L\oplus
L^\bot)=\rho_{|_P}(L)+\rho_{|_P}(L^\bot)\subset TP\times \Rr$. We
are left to prove that $F$ is closed with respect to the bracket
on $E$. Let $X, Y\in \Gamma(A)$ and $\al, \be\in \Gamma(A^*)$ such
that $X+\al$ and $Y+\be$ restricted to $P$ are sections of $F$. By
definition, the bracket on $E$ is given by
 \begin{align*}
(X+\al)\smc (Y+\be)  &= \brr{X,Y}_A+i_{\al} \d_* Y -i_{\be} \d_* X +\d_*\left(\al(Y)\right)+X_A(\al,\be, -)\\
& + \eval{\alpha,W}Y  -\eval{\be,W}X + \eval{\beta, X}W   \\
&+ \brr{\al,\be}_{*}+\Lie_{X} \be-i_Y \d \al + \eval{X,\phi}\be -
\eval{Y,\phi} \al+\eval{\al,Y}\phi\,
\end{align*}
where $W=\tilde \d_* 1$.

By hypothesis, we immediately have  that
$$\brr{X,Y}_{A|_P}=\brr{X_{|_P},Y_{|_P}}_L\in \Gamma(L), $$
$$\brr{\al,\be}_{A^*|_P}=\brr{\al_{|_P},\be_{|_P}}_{L^{\bot}}\in \Gamma(L^\bot), $$
$$\eval{\al,Y}_{|_P}\phi_{|_P}+\eval{\be,X}_{|_P} W_{|_P}=0,$$
$$\eval{\alpha,W}Y_{|P}-\eval{\be,W}X_{|P} + \eval{X,\phi}\be_{|P}- \eval{Y,\phi} \al_{|P} \in F $$
\mbox{and} $$X_A(\al,\be, -)_{|_P}\in   \Gamma(L).$$

Now, notice that $\al(Y)_{|_P}=0$, so $\d \al(Y)_{|_P}\in
\nu^*(P)=(TP)^0$. Since $\rho_{A^*|_P}(L^\bot)\subset TP$, we have
that
$$\d_{*} \al(Y)_{|_P}=\rho^*_{A^*|_P}d\al(Y)_{|_P}\in \Gamma(L).$$
Analogously, $\ds \d \be(X)_{|_P}=\rho^*_{A}d\be(X)_{|_P}\in
\Gamma(L^\bot).$

Also, $$\d_* Y(\al,\be)_{|_P}=\left(\rho_{A^*}(\al)\cdot \be(Y)-
\rho_{A^*}(\be)\cdot
\al(Y)-\brr{\al,\be}_{A^*}(Y)\right)_{|_P}=0,$$ so $i_{\al}\d_*
Y\in \Gamma(L)$. Analogously, $\ds i_\be\d_* X\in \Gamma(L)$ and
$\ds i_{X}\d \be, \, i_{Y}\d \al\in \Gamma(L^{\bot})$.

All these conditions allow us to say that $(X+\al)\smc
(Y+\be)_{|P} \in \Gamma(L\oplus L^\bot)$ and, consequently, $F$ is
a Dirac structure supported on $P$.
\end{proof}

\begin{cor}
Let $E=A\oplus A^*$ be the double of a Jacobi bialgebroid then
$F=L\oplus L^\bot$ is a Dirac structure supported on $P$ if and
only if $L$ and $L^\bot$ are Lie subalgebroids of $A$ and $A^*$.
\end{cor}

Notice that when $\phi=0$ and $W=0$, then $E=A\oplus A^*$ is the
double of a Lie bialgebroid. Then, if we have also $P=M$, i.e. the
Dirac structure is global, we obtain Proposition 7.1 of
\cite{LiuWeiXu}.

\begin{cor}
Let $ \mathcal{E}_{\omega}^1(M)=(TM\times \Rr)\oplus (T^*M\times
\Rr)$ be the standard Courant-Jacobi algebroid twisted by the
$3$-form $(\d \omega,\omega ) \in \Gamma(\wedge^{3}(T^*M \times
\Rr))$ and $i:P\hookrightarrow M$ a submanifold of $M$.
 Then
$F=(TP\times \Rr)\oplus (TP\times \Rr)^\bot$ is a Dirac structure of
$\mathcal{E}_{\omega}^1(M)$ supported on $P$ if and only if $i^*
\omega=0$.
\end{cor}

\begin{proof}
Since the Jacobi algebroid structure on $T^*M\times \Rr$ is the
null structure, we immediately have the three first conditions of
Theorem \ref{thm:dirac:Jacobi:direct:sum}. Condition 4) is ensured
by  $i^*\omega=0$ because, in this case, we also have $i^* \d
\omega=0$ and these two conditions are equivalent to $(\d
\omega,\omega)\in \Gamma(\wedge^3 (TP\times \Rr)^\bot)$.
\end{proof}

\begin{defn}
A \emph{Courant-Jacobi algebroid morphism} between two
Courant-Jacobi algebroids $E\to M$ and $E'\to M'$ is a Dirac
structure in $E\times \overline E'$ supported on $\graf \phi$,
where $\phi:M\to M'$ is a smooth map and $\overline E'$ denotes
the Courant-Jacobi algebroid obtained from $E'$ by changing the
sign of the bilinear form.
\end{defn}

If $E$ and $E'$ are in particular Courant algebroids, then a Dirac
structure in $E\times \overline E'$ supported on $\graf \phi$
defines a Courant algebroid morphism \cite{AleXu,BurIglSev}. Several
examples of Courant-Jacobi morphisms, extending the known ones for
Courant algebroids, appear naturally (see for instance
\cite{LiMein08}).

\begin{ex}
An orthogonal vector bundle isomorphism $\Phi:E\to E'$ covering a
diffeomorphism $\phi:M\to M'$ such that $\ds
\Phi\brr{X,Y}_E=\brr{\Phi X,\Phi Y}_{E'}$ and $\rho_{E'}\Phi(X)=(
\phi_*, id)\rho_E(X)$ defines an isomorphism between Courant-Jacobi
algebroids.
\end{ex}

\begin{ex}
A Courant-Jacobi morphism from $E$ to the zero Courant algebroid
over a point is the same as a Dirac structure of the Courant-Jacobi
algebroid $E$.
\end{ex}

\begin{ex}
Given a smooth map $\phi:M\to M'$ there is  a standard
Courant-Jacobi morphism between $\mathcal{E}^1(M)=(TM\times
\Rr)\oplus(T^*M\times \Rr)$ and $\mathcal{E}^1(M')=(TM'\times
\Rr)\oplus(T^*M'\times \Rr)$ given by
\[
L_\phi=\set{ \left((X,\phi^*f) + (\phi^*\al,\phi^*g)\, , \, (\phi_*
X, f) + (\al,g) \right)| \,
X\in\mathfrak{X}^1(M),\,\al\in\Omega^1(M'),\, f,g\in C^\infty(M')}.
\]
\end{ex}

\begin{ex}
Let $E\to M$ be a Courant-Jacobi algebroid and $\phi:M\to M\times M$
the diagonal map. The vector subbundle of $E\times\overline{E}\times
\overline{\mathcal{E}^1(M)}$ over $\graf \phi$,
\[L=\set{(X, X-\rho_E^*(\al,f), \rho_E(X)+ (\al,f))|\, X\in\Gamma(E), (\al,f)\in\Gamma(T^*M\times \Rr)}\]
is a Courant-Jacobi morphism between $E\times\overline{E}$ and
$\mathcal{E}^1(M)$.
\end{ex}

\begin{thm}\label{thm:quasi:Courant-Jacobi:morph}
Let $E_1=A\oplus A^*$ and $E_2=B\oplus B^*$ be doubles of
quasi-Jacobi bialgebroids $(A,\phi_A,\tilde\d_{*A},X_A)$ and
$(B,\phi_B,\tilde\d_{*B},X_B)$ and $(\Phi,\phi):A\to B$ a
quasi-Jacobi bialgebroid morphism, then
\[
F=\set{(a+\Phi^*b^*,\Phi a+ b^*)|\, a\in A_x \mbox{ and } b^*\in
B^*_{\phi(x)}, x\in M}\subset E_1\times \overline{E_2}
\]
is a Dirac structure supported on $\graf \phi$, i.e. $F$ is a
Courant-Jacobi algebroid morphism.
\end{thm}
\begin{proof}
Consider the following vector bundles over $M$ (or  $\graf
\phi\simeq M$)
$$L_{x}=\graf \Phi_x=\set{(a,\Phi_x a)|\, a\in A_x}\subset A_x\times B_{\phi(x)}$$
and
$$L_{x}^\bot=\set{(\Phi^*b^*, -b^*)|\, b^*\in B_{\phi(x)}^*}\subset A_x^*\times {B_{\phi(x)}^*}.$$

Since $\Phi$ is a Jacobi algebroid morphism, $L$  is clearly a Lie
subalgebroid of $A\times B$ (when seen as a vector bundle over
$\graf \phi\simeq M$). Analogously, we can also conclude that
$L^\bot$ is closed for the bracket on  $A^*\times \overline{B^*}$
(where $\overline{B^*}$ denotes the vector bundle $B^*$ with bracket
$-\br_{B^*}$) and it is compatible with the  anchor $\rho_{A^*\times
\overline{B^*}}=(\rho_{A^*},-\rho_{B^*})$. Since $\Phi X_A=X_B$, we
conclude that $L\oplus L^\bot$  is a Dirac structure  supported on
$\graf \phi$ of the double $(A\times B)\oplus (A^*\times
\overline{B^*})$ which is the Courant-Jacobi algebroid $(A\oplus
A^*)\times (B\oplus \overline{B^*})$. Finally, observe that the
vector bundle morphism $b+ b^*\mapsto b-b^*$ induces a canonical
isomorphism between $F$ and $L\oplus L^\bot$ and the result follows.
\end{proof}

\begin{cor}
Let $E_1=A\oplus A^*$ and $E_2=B\oplus B^*$ be doubles of
quasi-Lie bialgebroids $(A,\d_{*A},X_A)$ and $(B,\d_{*B},X_B)$ and
$(\Phi,\phi):A\to B$ a quasi-Lie bialgebroid morphism, then
\[
F=\set{(a+\Phi^*b^*,\Phi a+ b^*)|\, a\in A_x \mbox{ and } b^*\in
B^*_{\phi(x)}, x\in M}\subset E_1\times \overline{E_2}
\]
is a Courant algebroid morphism.
\end{cor}

\section{Jacobi-quasi-Nijenhuis algebroids}              %
\label{sec:Jacobi:quasi:Nij}                             %

Let $(A,{\phi})$ be a Jacobi algebroid. Recall that the induced
Lie algebroid structure from $A$ by $\phi$ on the vector bundle
$\hat A=A\times \Rr$ over $M\times \Rr$ is defined by the anchor
\begin{equation}\label{anchorhatA}
    \hat\rho(X)=\rho(X)+\eval{{\phi},X} \frac{\p}{\p t}, \quad
    X\in\Gamma(A),
\end{equation}
and the bracket induced by $\br$ for  time independent
multivectors
\begin{equation}\label{brackethatA}
    \brr{P,Q}_{\hat A}=\brr{P,Q}, \quad P,Q\in\Gamma(\wedge^{\bullet}A).
\end{equation}
The differential in $\hat A$ induced by $(A,{\phi})$ is denoted by
$\hat\d$. The differential obtained in $\hat A$ when ${\phi}=0$ is
denoted by $\tilde\d$.

Let $(A,{\phi})$ be a Jacobi algebroid and suppose that a Jacobi
bivector is given on $A$, i.e., a bivector
$\pi\in\Gamma(\wedge^{2} A)$ such that $ \brr{\pi,\pi}^{\phi}=0. $
It follows that $\tilde\pi=e^{-t} \pi$ is a Poisson bivector on
$\hat A$ and, consequently, it defines a Lie algebroid structure
on $\hat A^\ast$ (over $M\times \Rr$)  given by
\begin{equation}\label{dualstructure:hatAdual}
    \brr{\alpha,\beta}_{\tilde \pi}=\hat\Lie_{\tilde
    \pi^\sharp\alpha}\beta-\hat{\Lie}_{\tilde \pi^\sharp\beta}\alpha-\hat\d
    \tilde{\pi}(\alpha,\beta),
\end{equation}
\begin{equation}
\hat\rho_\ast(\alpha)=\hat\rho\circ\tilde \pi^\sharp(\alpha)
\end{equation}
where $\alpha,\beta\in\Gamma(\hat A^\ast)$ and $\hat{\Lie}$ is the
Lie derivative in $\hat A$. In particular, for $\alpha,\,
\beta\in\Gamma(A^\ast)$, we have
\begin{equation}\label{relgauging2}
    \brr{e^t\alpha,e^t\beta}_{\tilde \pi}=e^t(\Lie^{\phi}_{\pi^\sharp \alpha}\beta - \Lie^{\phi}_{\pi^\sharp\beta}\alpha-\d^{\phi}
    \pi(\alpha,\beta)).
\end{equation}
The Lie bracket
\begin{equation}\label{bracketdualA}
    \brr{\alpha,\beta}_{\pi}={\Lie}^{\phi}_{\pi^\sharp \alpha}\beta - {\Lie}^{\phi}_{\pi^\sharp\beta}\alpha-\d^{\phi}
    \pi(\alpha,\beta),
\end{equation}
together with the anchor
\begin{equation}\label{anchordualA}
    \rho_\ast=\rho\circ \pi^\sharp,
\end{equation}
endows  $A^\ast$ with a Lie algebroid structure over $M$.

In order to introduce Jacobi quasi-Nijenhuis algebroids we recall
now the notion of torsion of a vector bundle map in a Jacobi
algebroid.

Let $(A,\phi)$, with $A=(A,\br,\rho)$, be a Jacobi algebroid over a
manifold $M$.  The torsion  of a vector bundle map $N:A\to A$ (over
the identity) is defined by
\begin{equation}
\label{eq:Nijenhuis} \T_N(X,Y):=[NX,NY]-N[X,Y]_N, \quad  X,Y\in
\Gamma(A),
\end{equation}
where $\br_N$ is given by:
\[
[X,Y]_N:=[NX,Y]+[X,NY]-N[X,Y],\quad X,Y\in \Gamma(A).
\]

When $\T_N=0$, the vector bundle map $N$ is called a \emph{Nijenhuis
operator}, the triple $A_N=(A,\br_N,\rho_N=\rho\smc N)$ is a new Lie
algebroid and $\phi_1=N^*\phi$ is a $1$-cocyle of $A_N$ so that
$(A_N,\phi_1)$ is a new Jacobi algebroid. Finally,
$N:(A_N,\phi_1)\to (A,\phi)$ is an example of  Jacobi algebroid
morphism.

Now assume that also a Jacobi bivector $\pi\in\Gamma(\wedge^{2}
A)$ is given on $(A,\phi)$.

\begin{defn}
On a Jacobi algebroid $(A,{\phi})$ with a Jacobi bivector
$\pi\in\Gamma(\wedge^2 A)$, we say that a vector bundle map $N:A\to
A$ is \emph{compatible} with $\pi$ if $N\pi^\sharp=\pi^\sharp N^*$
and the \emph{Magri-Morosi concomitant} vanishes:
\[\C(\pi,N)(\al,\be)= \brr{\al,\be}_{N\pi}-\brr{\al,\be}_\pi^{N^\ast}=0,\]
where $\br_{N\pi}$ is the  bracket defined by the bivector field
$N\pi\in\Gamma(\wedge^{2} A)$, and $\br_\pi^{N^\ast}$ is the Lie
bracket obtained from the Lie bracket $\br_\pi$ by deformation
along the  tensor $N^\ast$.
\end{defn}

\begin{defn}\label{def:Jacobi:quasi:Nij:alg}
A \emph{Jacobi quasi-Nijenhuis algebroid} $(A, \phi, \pi, N,
\varphi)$ is a Jacobi algebroid $(A,{\phi})$ equipped with a Jacobi
bivector $\pi$, a vector bundle map $N:A\to A$ compatible with $\pi$
and a   3-form $\varphi\in\Gamma(\wedge^{3} A^*)$ such that
$\d^{\phi} \varphi=0$  (i.e. $\varphi$ is $\d^{\phi}$-closed), and
\[
\T_N(X,Y)=-\pi^\sharp\left(i_{X\wedge Y} \varphi \right)\quad
\mbox{ and } \quad  \d^{\phi} \left(i_N\varphi\right)=0,
\]
where $i_N$ is the degree-zero derivation of $\Gamma( \wedge
^\bullet A^*)$ given, for the case of a $3$-form $\varphi$, by
$$i_N \varphi(X,Y,Z)=\varphi(NX,Y,Z)+\varphi(X,NY,Z)+\varphi(X,Y,NZ).$$
\end{defn}

In particular, when  $\phi=0$, we have a \emph{Poisson
quasi-Nijenhuis Lie algebroid} $(A, \pi, N, \varphi)$  \cite{Ant,
StiXu}.

\begin{prop}\label{prop:Jac_quasi_Nij:Pois_quasi_Nij}
If $(A, \phi, \pi, N, \varphi)$ is a Jacobi quasi-Nijenhuis
algebroid, then its Poissonization $(\hat A, e^{-t}\pi, N,
e^t\varphi)$ is a Poisson quasi-Nijenhuis Lie algebroid.
\end{prop}

\begin{proof}
It is well known that $(\hat A, e^{-t}\pi)$ is a Poisson Lie
algebroid over $M\times \Rr$. Also we have that
\begin{equation*}
\hat \d (e^t\varphi)=e^t \d^\phi\varphi=0,
\end{equation*}
\begin{equation*}
\hat\d(i_Ne^t\varphi)=\hat\d(e^ti_N\varphi)=e^t\d^\phi
(i_N\varphi)=0
\end{equation*}
and
\begin{align*}
\T_N(X,Y)&=-\pi^\sharp\left(i_{X\wedge Y}\varphi \right)\\
&=-(e^{-t}\pi)^\sharp\left(i_{X\wedge Y} e^t\varphi \right),
\end{align*}
and we conclude that $(\hat A, e^{-t}\pi, N, e^t\varphi)$ is a
Poisson quasi-Nijenhuis Lie algebroid.
\end{proof}

When $(A, \phi)=(TM \times \Rr, (0,1))$ and ${\mathcal N} : TM
\times \Rr \to TM \times \Rr$ is the $C^\infty(M)$-linear map
given, for any section $(X,f)$ of $\Gamma(TM \times \Rr)$, by
${\mathcal N}(X,f)=(NX + fY, i_X \gamma+fg)$, where $N$ is a
$(1,1)$-tensor field on $M$, $Y \in {\mathfrak X}(M)$, $\gamma \in
\Omega^{1}(M)$ and $g \in C^\infty(M)$, we may introduce the
notion of Jacobi quasi-Nijenhuis manifold.

\begin{defn}\label{def:Jacobi:quasi:Nij:mfd}
A \emph{Jacobi quasi-Nijenhuis manifold} is a Jacobi manifold
$(M,\Lambda,E)$ together with a $C^\infty(M)$-linear map
${\mathcal N}:= (N,Y, \gamma, g) : TM \times \Rr \to TM \times
\Rr$  compatible with $(\Lambda,E)$ and a $2$-form $\omega \in
\Omega^{2}(M)$ such that
\[
\T_{\mathcal N}=-(\Lambda,E)^\sharp \circ (\d \omega, \omega) \;
\mbox{ and } \;  \d^{(0,1)} \left(i_{\mathcal N}(\d \omega,
\omega)\right)=0.
\]
\end{defn}

In the  case where $\omega=0$, $(M, (\Lambda,E), {\mathcal N})$ is
a \emph{Jacobi-Nijenhuis manifold} \cite{jf-lsjn}.

\

Let us now consider the Poissonization $(\hat{M}, \hat{\Lambda})$
of the Jacobi manifold $(M,\Lambda,E)$, i.e. $\hat{M}=M \times
\Rr$ and $\hat{\Lambda}= e^{-t} (\Lambda +
\frac{\partial}{\partial t} \wedge E)$, and the tensor field
$\hat{N}$ of type $(1,1)$ on $\hat{M}$ given by
$$\hat{N}= N+ Y \otimes \d t + \frac{\partial}{\partial t} \otimes \gamma + g \frac{\partial}{\partial t} \otimes \d t.$$

\begin{prop}\label{prop:Jac_quasi_Nij_mfd:Pois_quasi_Nij_mfd}
The quadruple $(M,(\Lambda,E), {\mathcal N}, \omega)$ is a Jacobi
quasi-Nijenhuis manifold if and only if $(\hat{M}, \hat{\Lambda},
\hat{N}, \hat{\d}(e^t \omega))$ is a Poisson quasi-Nijenhuis
manifold.
\end{prop}

\begin{proof}
It is known \cite{jf-lsjn} that ${\mathcal N}$ is compatible with
$(\Lambda,E)$ if and only if $\hat{N}$ is compatible with
$\hat{\Lambda}$. Moreover, a direct computation shows that
$$\T_{\mathcal N}((X_1,f_1),(X_2,f_2))=- (\Lambda,E)^\sharp \left( (\d \omega, \omega)((X_1,f_1),(X_2,f_2), -)\right)$$
is equivalent to
$$\T_{\hat{ N}}\left(X_1+f_1 \frac{\partial}{\partial t},\, X_2+f_2\frac{\partial}{\partial t}\right)=- \hat{\Lambda}^\sharp\left(\hat{\d}(e^t
\omega)\left(X_1+f_1 \frac{\partial}{\partial t}, \,
X_2+f_2\frac{\partial}{\partial t}, -\right) \right),$$ for all
$X_i \in {\mathfrak X}^1(M)$, $f_i \in C^\infty(M), i=1,2$.

Finally, $$i_{\mathcal N}(\d \omega, \omega)=(i_N \d\omega, i_Y \d
\omega)+ (\gamma \wedge \omega, g \omega)+ (0, i_N \omega)$$ and
since

\begin{eqnarray*}
\lefteqn{\hat{\d}(i_{\hat{N}}\hat{\d}(e^t \omega))\left( X_1+f_1
\frac{\partial}{\partial t},X_2+f_2\frac{\partial}{\partial t},
X_3+f_3\frac{\partial}{\partial t}\right)=}\\
& =& e^t ( \d^{(0,1)}(i_N \d\omega, i_Y \d \omega) + \d^{(0,1)}(\gamma \wedge \omega, g \omega) \\
& & - (0,1) \wedge (\d i_N \omega, -\d i_Y
\omega))((X_1,f_1),(X_2,f_2) ,(X_3,f_3)),
\end{eqnarray*}
  the proof is completed.
\end{proof}

The main aim of this section is to prove that for any Jacobi
quasi-Nijenhuis algebroid there is an associated quasi-Jacobi
bialgebroid structure on its dual.

\begin{thm}\label{theor:quasi:Jacobi:bialg}
If $(A, \phi, \pi, N, \varphi)$ is a Jacobi quasi-Nijenhuis
algebroid, then $(A_\pi^*, W, \d_N^{N^*\phi}, \varphi)$ is a
quasi-Jacobi bialgebroid, where $W=-\pi^\sharp(\phi)$.
\end{thm}

\begin{proof}
From Proposition \ref{prop:Jac_quasi_Nij:Pois_quasi_Nij} we know
that $(\hat A, \tilde\pi, N, e^t\varphi)$, with
$\tilde\pi=e^{-t}\pi$, is a Poisson quasi-Nijenhuis algebroid.
 By \cite[Theorem 3.2]{Ant} it follows that $(\hat
A^\ast_{\tilde\pi}, \hat\d_N, e^t\varphi)$ is a quasi-Lie
bialgebroid. Notice that $\hat d_N$ is the derivation associated to
$\hat A_N$, i.e., to the deformation by the tensor $N$ of the Lie
algebroid $\hat A$. Consider now $A_N$, i.e., the deformation by the
tensor $N$ of the Lie algebroid $A$. Since $N$ is not Nijenhuis,
then $A_N$ is not necessarily a Lie algebroid. Nevertheless, the
procedure described in \eqref{anchorhatA}, \eqref{brackethatA} can
be applied to $(A_N,N^*\phi)$ obtaining a new bracket
$\br_{\widehat{A_N}}$ and a new anchor $\widehat{\rho_N}$ on $\hat
A$ or, equivalently, a differential $\widehat{\d_N}$. This
differential coincides with $\hat d_N$ because, for
$X,Y\in\Gamma(A)$, we have
\begin{align*}
[X,Y]_{\hat{A}_N}:&=[NX,Y]_{\hat{A}}+[X,NY]_{\hat{A}}-N[X,Y]_{\hat{A}}\\
                &=[NX,Y]_A+[X,NY]_A-N[X,Y]_A=[X,Y]_{A_N}=[X,Y]_{\widehat{A_N}}
\end{align*}
and
\begin{align*}
{\hat{\rho}}_N(X):&={\hat{\rho}}\smc N(X)= \rho(NX) + \eval{\phi,NX}\frac{\p}{\p t} \\
                &=\rho \smc N(X) + \eval{N^*\phi,X}\frac{\p}{\p t}\\
                &={\widehat{\rho_N}}(X).
\end{align*}

Thus $(\hat A^\ast_{\tilde\pi}, \hat\d_N, e^t\varphi)=(\hat
A^\ast_{\tilde\pi}, \widehat{\d_N}, e^t\varphi)$ and this is known
to  induce a quasi-Jacobi bialgebroid structure on $A^*$: $(A_\pi^*,
W, \d_N^{N^*\phi}, \varphi)$ (see  \cite[Theorem 4.1]{jf-qjjq}).
\end{proof}

For the case of Jacobi quasi-Nijenhuis manifolds we have an
equivalence:

\begin{prop} The quadruple $(M,(\Lambda,E), {\mathcal N}, \omega)$ is a Jacobi quasi-Nijenhuis manifold
if and only if $((T^\ast M \times \Rr)_{(\Lambda,E)}, (-E,0),
\d^{(\gamma, g)}_{\mathcal N}, (\d \omega, \omega))$ is a
quasi-Jacobi bialgebroid.
\end{prop}

\begin{proof}
From Proposition \ref{prop:Jac_quasi_Nij_mfd:Pois_quasi_Nij_mfd},
$(M,(\Lambda,E), {\mathcal N}, \omega)$ is a Jacobi
quasi-Nijenhuis manifold if and only if $(\hat{M}, \hat{\Lambda},
\hat{N}, e^t \omega)$ is a Poisson quasi-Nijenhuis manifold, which
is equivalent to $((T^\ast \hat{M})_{\hat{\Lambda}},
\hat{\d}_{\hat{N}}, e^t \omega)$ being a quasi-Lie bialgebroid
over $\hat{M}$ \cite{StiXu}.  This, in turn, is equivalent to
$((T^\ast M \times \Rr)_{(\Lambda,E)}, (-E,0), \d^{(\gamma,
g)}_{\mathcal N}, (\d \omega, \omega))$ being a quasi-Jacobi
bialgebroid over $M$ \cite{jf-qjjq}. Notice that $(\Lambda,
E)^\sharp(0,1)= (E,0)$ and $ {\mathcal N}^\ast(0,1)=(\gamma,g)$.

\end{proof}

  Suppose $(A,\phi,\pi,N,\varphi)$ is a Jacobi quasi-Nijenhuis
algebroid. The double of the quasi-Jacobi bialgebroid $(A_\pi^*,
W, \d_N^{N^*\phi}, \varphi)$ is a Courant-Jacobi algebroid (see
Example \ref{double:quasi:Jacobi:bialgebroid}) that we denote by
$E_\pi^\varphi$.

An interesting case is when the $3$-form $\varphi$ is  the image
by $N^*$ of another $\d^{\phi}$-closed $3$-form $\psi$:
\[
\varphi=N^*\psi \quad \mbox{and} \quad \d^\phi\psi=0.
\]
In this case $(A, \phi, N\pi, \psi )$ is a twisted Jacobi
algebroid because
\[\brr{N\pi,N\pi}^\phi=2\pi^\sharp(\varphi)=2\pi^\sharp(N^*\psi)=2N\pi^\sharp(\psi)\]
and $A^*$  has a structure  of Lie algebroid:
$A^{*\psi}_{N\pi}=(A^*, \br_{N\pi}^\psi, N\pi^\sharp )$ with
1-cocycle $W_1=-N\pi^\sharp(\phi)$ (see \cite{jf-qjjq}).

 Equipping $A$ with
the differential $\d^\prime$ given by
\[
\d^\prime f=\d f, \quad \mbox{and} \quad
\d^\prime\al=\d\al-i_{N\pi^\sharp \al}\psi,
\]
  for  $f\in C^\infty(M)$ and $\al\in\Gamma(A^*)$,
we obtain a quasi-Jacobi bialgebroid: $(A^{*\psi }_{N\pi}, W,
\d^{\prime,\, \phi}, \psi)$ \cite{jf-qjjq}. Its double is a
Courant-Jacobi algebroid and we denote it by $E_{N\pi}^\psi$.

\begin{thm}
Let $(A,\phi, \pi, N, \varphi)$ be a Jacobi quasi-Nijenhuis
algebroid and suppose that $\varphi=N^*\psi$, for some
$\d^{\phi}$-closed $3$-form $\psi$, then
 \[
 F=\set{\left(a+N^*\al, Na+\al \right)|\; a\in A \mbox{ and } \al \in A^* } \subset E_{N\pi}^\psi\times \overline{E_\pi^\varphi}
 \]
defines a Courant-Jacobi algebroid morphism between
$E_{N\pi}^\psi$ and $E_\pi^\varphi$.
\end{thm}

In order to prove the theorem, we need to remark the following
property.

\begin{lem}
Let $(A,\phi, \pi, N, \varphi)$ be a Jacobi quasi-Nijenhuis Lie
algebroid, then
\[
\eval{\T_{N^*}(\al,\be),X}=\varphi(\pi^\sharp\al, \pi^\sharp\be,
X),
\]
for all $X\in\Gamma(A)$ and  $\al,\be\in\Gamma(A^*)$.
\end{lem}

\begin{proof}
The compatibility between $N$ and $\pi$ implies that (see
\cite{KosMagri})
\begin{align*}
\eval{\T_{N^*}(\al,\be),X}&=\eval{\al, \T_N(X, \pi^\sharp\be)},
\end{align*}
so
\begin{align*}
\eval{\T_{N^*}(\al,\be),X}&=\eval{\al, -\pi^\sharp\left(i_{X\wedge
\pi^\sharp\be}\varphi\right)}=-\varphi(X,\pi^\sharp\be,\pi^\sharp
\al)=\varphi(\pi^\sharp \al,\pi^\sharp \be,X).
\end{align*}
\end{proof}

\begin{proof}[Proof of the Theorem]
First notice that $N^*: A^{*\psi}_{N\pi}\to A^*_{\pi}$ is a Jacobi
algebroid morphism because it is obviously compatible with the
anchors, $NW=-N\pi^\sharp\phi=W_1$ and
\begin{align*}
N^*\brr{\al,\be}_{N\pi}^\psi&=N^*\brr{\al, \be}_{N\pi} + N^*\psi(\pi^\sharp \al, \pi^\sharp \be, - )\\
&= \brr{N^*\al, N^* \be}_\pi - \T_{N^*}(\al,\be) +
\varphi(\pi^\sharp \al, \pi^\sharp \be, -)= \brr{N^*\al, N^*
\be}_\pi.
\end{align*}

Let $\br^\prime$ be the bracket on the sections of $A$ induced by
the differential $\d^\prime$. First notice that $N^*$ preserves
all the cocycles. Also
$$\brr{X,f}^\prime=\eval{\d^\prime f, X}= \eval{\d f, X}, $$
so $\ds N^*\d^\prime f=\d_N f$, for all $f\in C^\infty(M)$ and
$X\in \Gamma(A)$. And since
\[
\brr{X,Y}^\prime= \brr{X,Y} - (N\pi)^\sharp(\psi(X,Y,  - )),
\]
we have:
\begin{align*}
N\brr{X,Y}_N&= \brr{NX, NY} - \T_N(X,Y)= \brr{NX,NY} + \pi^\sharp(i_{X\wedge Y}N^*\psi)\\
&= \brr{NX,NY} + \psi (NX, NY, N\pi^\sharp - )=
\brr{NX,NY}^\prime,
\end{align*}
for all $X,Y\in \Gamma(A)$.

In this way we may conclude that  $N^*:A^{*\psi}_{N\pi}\to
A^*_{\pi}$ is a quasi-Jacobi bialgebroid morphism (see Definition
\ref{def:quasi:Jacobi:morphism}) and the result follows from
Theorem \ref{thm:quasi:Courant-Jacobi:morph}.
\end{proof}

\section*{Acknowledgments}
This work was partially supported by CMUC/FCT. The first and third
author also acknowledge partial financial support by projects
PTDC/MAT/69635/2006 and PTDC/MAT/099880/2008. The second author
acknowledges partial financial support by a FCT/Ci\^{e}ncia 2008
research fellowship and a Juan de la Cierva postdoctoral grant.


\begin{thebibliography}{99}

\bibitem{Ant} Antunes, P.: Poisson quasi-Nijenhuis structures with
background. \emph{Lett. Math. Phys.} \textbf{86} (2008), 33--45.

\bibitem{AleXu} Alekseev, A. and Xu, P.: Derived brackets and
Courant algebroids. \emph{Unpublished manuscript}.

\bibitem{BurIglSev} Bursztyn, H., Iglesias Ponte, D. and
\v{S}evera, P.: Courant morphisms and moment maps. \emph{Math.
Res. Lett.} \textbf{16} (2009), 215--232.

\bibitem{rj}Caseiro, R., Nunes da Costa, J. M.:  Jacobi-Nijenhuis algebroids
and their modular classes. \emph{J. Phys. A: Math. Theor.}
\textbf{40} (2007), 13311--13330.


\bibitem{GraMar1} Grabowski J. and Marmo G.: Jacobi structures revisited.
\emph{J.~Phys.~A:~Math.~Gen.} \textbf{34} (2001), 10975--10990.

\bibitem{GraMar} Grabowski J. and Marmo G.: The graded Jacobi
algebras and (co)homology. \emph{J.~Phys.~A:~Math.~Gen.}
\textbf{36} (2003), 161--181.

\bibitem{GraUrb} Grabowski, J. and Urbanski, P.: Lie algebroids and
Poisson-Nijenhuis structures. \emph{Rep. Math. Phys.} \textbf{40}
(1997), 195--208.

\bibitem{Igl} Iglesias-Ponte, D., Lie Groups and Groupoids and Jacobi structures, Ph.D. thesis, 2003.

\bibitem{IglMarr} Iglesias-Ponte, D. and Marrero J.C.: Generalized Lie bialgebroids and Jacobi
structures, \emph{J. Geom. Phys.} \textbf{40} (2001), 176--199.

\bibitem{KosMagri} Kosmann-Schwarzbach, Y. and Magri, F.:
Poisson-Nijenhuis structures. \emph{Ann.~Inst.~Henri Poincar\'e}
\textbf{53} (1990), 35--81.

\bibitem{ks3}{Kosmann-Schwarzbach, Y.: Quasi, twisted, and all
that $\ldots$ In Poisson geometry and Lie algebroid theory. In:
Marsden, J.E., Ratiu, T.S. (eds.) The breadth of symplectic and
Poisson geometry. Progress in Mathematics, vol. 232, pp. 363-389.
Birkhauser, Boston (2005)}


\bibitem{LiMein08} Li-Bland, D. and Meinrenken, E.: Courant algebroids and Poisson geometry.
\emph{International Mathematics Research Notices} (2009).

\bibitem{LiuWeiXu} Liu, Z.-J., Weinstein, A. and Xu, P.: Manin triples
for Lie bialgebroids.  \emph{J. Diff. Geom.}  \textbf{45} (1997),
547--574.

\bibitem{jj} Nunes da Costa, J.M., Clemente-Gallardo J.: Dirac structures for generalized Lie
bialgebroids. \emph{J. Phys. A: Math. Gen.} \textbf{37} (2004),
2671--2692.

\bibitem{jf-tj} Nunes da Costa, J.M., Petalidou, F.: Twisted
Jacobi manifolds, twisted Dirac-Jacobi structures and quasi-Jacobi
bialgebroids. \emph{J. Phys. A: Math. Gen.} \textbf{39} (2006),
10449--10475.

\bibitem{jf-qjjq} Nunes da Costa, J.M., Petalidou, F.: On quasi-Jacobi
and Jacobi-quasi bialgebroids. \emph{Lett.  Math. Phys.}
\textbf{80} (2007), 155--169.

\bibitem{jf-lsjn} Petalidou, F., Nunes da Costa, J.M.:
Local structure of Jacobi-Nijenhuis manifolds. \emph{J. Geom.
Phys.} \textbf{45} (2003), 323--367.

\bibitem{Roy} Roytenberg, D.: Quasi-Lie bialgebroids and twisted
Poisson manifolds. \emph{Lett.  Math. Phys.} \textbf{61} (2002),
123--137.

\bibitem{Sheng} Sheng, Y.: Jacobi quasi-Nijenhuis algebroids.
\emph{arXiv:0903.4332}.

\bibitem{StiXu} Sti\'enon, M. and Xu, P.: Poisson Quasi-Nijenhuis
Manifolds. \emph{Comm. Math. Phys.} \textbf{50} (2007), 709--725.

\end{thebibliography}
\end{document}